\documentclass{amsart}
\usepackage{amsmath,amssymb,amsfonts}
\usepackage{color}
\usepackage[colorlinks=true,linkcolor=blue,urlcolor=blue,citecolor   = red]{hyperref}

\newtheorem{theorem}{Theorem}[section]
\newtheorem{corollary}[theorem]{Corollary}
\newtheorem{definition}[theorem]{Definition}

\newtheorem{proposition}[theorem]{Proposition}
\newtheorem{remark}[theorem]{Remark}

\begin{document}
\title{On the class of weak almost limited operators}
\author[A. Elbour]{Aziz Elbour}
\address[A. Elbour]{Universit\'{e} Moulay Isma\"{\i}l, Facult\'{e} des
Sciences et Techniques, D\'{e}partement de Math\'{e}\-matiques, B.P. 509,
Errachidia, Morocco.}
\email{azizelbour@hotmail.com}
\author[N. Machrafi]{Nabil Machrafi}
\address[N. Machrafi]{Universit\'{e} Ibn Tofail, Facult\'{e} des Sciences, D%
\'{e}partement de Math\'{e}\-matiques, B.P. 133, K\'{e}nitra, Morocco.}
\email{nmachrafi@gmail.com}
\author[M. Moussa]{Mohammed Moussa}
\address[M. Moussa]{Universit\'{e} Ibn Tofail, Facult\'{e} des Sciences, D%
\'{e}partement de Math\'{e}\-matiques, B.P. 133, K\'{e}nitra, Morocco.}
\email{mohammed.moussa09@gmail.com}

\begin{abstract}
We introduce and study the class of weak almost limited operators. We
establish a characterization of pairs of Banach lattices $E$, $F$ for which
every positive weak almost limited operator $T:E\rightarrow F$ is almost
limited (resp. almost Dunford-Pettis). As consequences, we will give some
interesting results.
\end{abstract}

\maketitle

\section{Introduction}

Throughout this paper $X,$ $Y$ will denote real Banach spaces, and $E,\,F$
will denote real Banach lattices. $B_{X}$ is the closed unit ball of $X$%
\textbf{\ }and $\mathrm{sol}\left( A\right) $ denotes the solid hull of a
subset $A$ of a Banach lattice. We will use the term operator $%
T:X\rightarrow Y$ between two Banach spaces to mean a bounded linear mapping.

Let us recall that a norm bounded subset $A$ of $X$ is called a \emph{%
Dunford-Pettis set} (resp.a \emph{limited set}) if each weakly null sequence
in $X^{\ast }$ (resp. weak* null sequence in $X^{\ast }$) converges
uniformly to zero on $A$. An operator $T:X\rightarrow Y$ is called \emph{%
Dunford-Pettis} if $x_{n}\overset{w}{\rightarrow }0$ in $X$ implies $%
\left\Vert Tx_{n}\right\Vert \rightarrow 0$, equivalently, if $T$ carries
relatively weakly compact subsets of $X$ onto relatively compact subsets of $%
Y$. An operator $T:X\rightarrow Y$ is said to be \emph{limited} whenever $%
T\left( B_{X}\right) $ is a limited set in $Y$, equivalently, whenever $%
\left\Vert T^{\ast }\left( f_{n}\right) \right\Vert \rightarrow 0$ for every
weak* null sequence $\left( f_{n}\right) \subset Y^{\ast }$.

Aliprantis and Burkinshaw \cite{AB} introduced the class of weak
Dunford--Pettis operators. An operator $T:X\rightarrow Y$ is said to be 
\emph{weak Dunford--Pettis} whenever $x_{n}\overset{w}{\rightarrow }0$ in $X$
and $f_{n}\overset{w}{\rightarrow }0$ in $Y^{\ast }$ imply $f_{n}\left(
Tx_{n}\right) \rightarrow 0$, equivalently, whenever $T$ carries weakly
compact subsets of $X$ to Dunford--Pettis subsets of $Y$ \cite[Theorem 5.99]%
{AB3}. Next H'michane et al. \cite{Hmi} introduced the class of weak*
Dunford-Pettis operators, and characterized this class of operators and
studied some of its properties in \cite{Kad}. An operator $T:X\rightarrow Y$
is called \emph{weak* Dunford-Pettis} whenever $x_{n}\overset{w}{\rightarrow 
}0$ in $X$ and $f_{n}\overset{w^{\ast }}{\rightarrow }0$ in $Y^{\ast }$
imply $f_{n}\left( Tx_{n}\right) \rightarrow 0$, equivalently, whenever $T$
carries relatively weakly compact subsets of $X$ onto limited subsets of $Y$ 
\cite[Theorem 3.2]{Kad}.

Recently, two classes of norm bounded sets are considered in the theory of
Banach lattices. From \cite{Bouras} (resp. \cite{chen}), a norm bounded
subset $A$ of a Banach lattice $E$ is said to be an \emph{almost
Dunford-Pettis set} (resp. \emph{an almost limited set}), if every disjoint
weak null (resp. weak* null) sequence $(f_{n})$ in $E^{\ast }$ converges
uniformly to zero on $A$. Clearly, all Dunford-Pettis sets (resp. limited
sets) in a Banach lattice are almost Dunford-Pettis (resp. almost limited).
Also, every almost limited set is almost Dunford-Pettis. But the converse
does not hold in general.

Let us recall that an operator $T:E\rightarrow X$ is said to be \emph{almost
Dunford-Pettis}, if $T$ carries every disjoint weakly null sequence to a
norm null sequence, or equivalently, if $T$ carries every disjoint weakly
null sequence consisting of positive terms to a norm null sequence \cite[%
Remark 1]{WN3}. From \cite{Mach}, an operator $T:X\rightarrow E$ is called 
\emph{almost limited} whenever $T\left( B_{E}\right) $ is an almost limited
set in $E$, equivalently, whenever $\left\Vert T^{\ast }\left( f_{n}\right)
\right\Vert \rightarrow 0$ for every disjoint weak* null sequence $\left(
f_{n}\right) \subset E^{\ast }$.

Using the almost Dunford-Pettis sets, Bouras and Moussa \cite{Bouras2}
introduced the class of weak almost Dunford-Pettis operators. An operator $%
T:X\rightarrow E$ is called \emph{weak almost Dunford-Pettis} operator
whenever $T$ carries relatively weakly compact subsets of $X$ to almost
Dunford--Pettis subsets of $E$, equivalently, whenever $f_{n}(T(x_{n}))%
\rightarrow 0$ for all weakly null sequences $(x_{n})$ in $X$ and for all
weakly null sequences $(f_{n})$ in $E^{\ast }$ consisting of pairwise
disjoint terms \cite[Theorem 2.1]{Bouras2}.

In this paper, using the almost limited sets, we introduce the class of 
\emph{weak almost limited} operators $T:X\rightarrow E$, which carries
relatively weakly compact subsets of $X$ to almost limited subsets of $E$
(Definition \ref{def}). It is a class which contains that of weak*
Dunford-Pettis (resp. almost limited). We establish some characterizations
of weak almost limited operators. After that, we derive the domination
property of this class of operators (Corollary \ref{domination}). Next, we
characterize pairs of Banach lattices $E$, $F$ for which every positive weak
almost limited operator $T:E\rightarrow F$ is almost limited (resp. almost
Dunford-Pettis). As consequences, we will give some interesting results.

To show our results we need to recall some definitions that will be used in
this paper. A Banach lattice $E$ has

\begin{itemize}
\item[$\emph{-}$] the Dunford-Pettis property, if $x_{n}\overset{w}{%
\rightarrow }0$ in $E$ and $f_{n}\overset{w}{\rightarrow }0$ in $E^{\ast }$
imply $f_{n}\left( x_{n}\right) \rightarrow 0$ as $n\rightarrow \infty $,
equivalently, each relatively weakly compact subset of $E$ is
Dunford--Pettis.

\item[$\emph{-}$] the Dunford-Pettis* property (DP* property for short), if $%
x_{n}\overset{w}{\rightarrow }0$ in $E$ and $f_{n}\overset{w^{\ast }}{%
\rightarrow }0$ in $E^{\ast }$ imply $f_{n}\left( x_{n}\right) \rightarrow 0$%
, equivalently, each relatively weakly compact subset of $E$ is limited.

\item[$\emph{-}$] the weak Dunford-Pettis* property (wDP* property), if $%
f_{n}\left( x_{n}\right) \rightarrow 0$ for every weakly null sequence $%
\left( x_{n}\right) $ in $E$ and for every disjoint weak* null sequence $%
\left( f_{n}\right) $ in $E^{\ast }$, equivalently, each relatively weakly
compact subset of $E$ is almost limited \cite[Definition 3.1]{chen}.

\item[$\emph{-}$] the Schur (resp. positive Schur) property, if $\left\Vert
x_{n}\right\Vert \rightarrow 0$ for every weak null sequence $\left(
x_{n}\right) \subset E$ (resp. $\left( x_{n}\right) \subset E^{+}$).

\item[$\emph{-}$] the positive dual Schur property, if $\left\Vert
f_{n}\right\Vert \rightarrow 0$ for every weak* null sequence $\left(
f_{n}\right) \subset \left( E^{\ast }\right) ^{+}$, equivalently, $%
\left\Vert f_{n}\right\Vert \rightarrow 0$ for every weak* null sequence $%
\left( f_{n}\right) \subset \left( E^{\ast }\right) ^{+}$ consisting of
pairwise disjoint terms \cite[Proposition 2.3]{WN2012}.

\item[$\emph{-}$] the property (d) whenever $\left\vert f_{n}\right\vert
\wedge \left\vert f_{m}\right\vert =0$ and $f_{n}\overset{w^{\ast }}{%
\rightarrow }0$ in $E^{\ast }$ imply $\left\vert f_{n}\right\vert \overset{%
w^{\ast }}{\rightarrow }0$.
\end{itemize}

It should be noted, by Proposition 1.4 of \cite{WN2012}, that every $\sigma $%
-Dedekind complete Banach lattice has the property $\mathrm{(d)}$ but the
converse is not true in general. In fact, the Banach lattice $\ell ^{\infty
}/c_{0}$\ has the property $\mathrm{(d)}$ but it is not $\sigma $-Dedekind
complete \cite[Remark 1.5]{WN2012}.

Our notions are standard. For the theory of Banach lattices and operators,
we refer the reader to the monographs \cite{AB3, MN}.

\section{Main results}

We start this section by the following definition.

\begin{definition}
\label{def}An operator $T:X\rightarrow E$ from a Banach space $X$ into a
Banach lattice $E$ is called weak almost limited if $T$ carries each
relatively weakly compact set in $X$ to an almost limited set in $E$.
\end{definition}

Clearly, a Banach lattice $E$ has the DP* property (resp. wDP* property) if
and only if the identity operator $I:E\rightarrow E$ is weak* Dunford-Pettis
(resp. weak almost limited). Also, every weak* Dunford-Pettis (resp. almost
limited) operator $T:X\rightarrow E$ is weak almost limited, but the
converse is not true in general. In fact, the identity operator $I:L^{1}%
\left[ 0,1\right] \rightarrow L^{1}\left[ 0,1\right] $ (resp. $I:\ell
^{1}\rightarrow \ell ^{1}$) is weak almost limited but it fail to be weak*
Dunford-Pettis (resp. almost limited).

In terms of weakly compact and almost limited operators the weak almost
limited operators are characterized as follows.

\begin{theorem}
\label{aw*DP}For an operator $T:X\rightarrow E$, the following assertions
are equivalents:

\begin{enumerate}
\item $T$ is weak almost limited.

\item If $S:Z\rightarrow X$ is a weakly compact operator, where $Z$ is an
arbitrary Banach space, then the operator $T\circ S$ is almost limited.

\item If $S:\ell ^{1}\rightarrow X$ is a weakly compact operator then the
operator $T\circ S$ is almost limited.

\item For every weakly null sequence $\left( x_{n}\right) \subset X$ and
every disjoint weak* null sequence $\left( f_{n}\right) \subset E^{\ast }$
we have $f_{n}\left( Tx_{n}\right) \rightarrow 0$.
\end{enumerate}
\end{theorem}

\begin{proof}
$\left( 1\right) \Rightarrow \left( 2\right) $ Let $\left( f_{n}\right)
\subset E^{\ast }$ be a disjoint weak* null sequence. We shall proof that $%
\left\Vert \left( T\circ S\right) ^{\ast }\left( f_{n}\right) \right\Vert
\rightarrow 0$. Otherwise, by choosing a subsequence we may suppose that
there is $\varepsilon $ with $\left\Vert \left( T\circ S\right) ^{\ast
}\left( f_{n}\right) \right\Vert >\varepsilon >0$ for all $n\in \mathbb{N}$.
So for every $n$ there exists some $x_{n}\in B_{Z}$ with $\left( T\circ
S\right) ^{\ast }\left( f_{n}\right) \left( x_{n}\right) =f_{n}\left(
T\left( S\left( x_{n}\right) \right) \right) \geq \varepsilon $.

On the other hand, as $S$ is weakly compact, $\left\{ S\left( x_{k}\right)
:k\in \mathbb{N}\right\} $ is a relatively weakly compact subset of $X$ .
Then $\left\{ T\left( S\left( x_{k}\right) \right) :k\in \mathbb{N}\right\} $
is an almost limited set in $E$ (as $T$ is weak almost limited). So $\sup
\left\{ \left\vert f_{n}\left( T\left( S\left( x_{k}\right) \right) \right)
\right\vert :k\in \mathbb{N}\right\} \rightarrow 0$ as $n\rightarrow \infty $%
. But this implies that $f_{n}\left( T\left( S\left( x_{n}\right) \right)
\right) \rightarrow 0$, which is impossible. Thus, $\left\Vert \left( T\circ
S\right) ^{\ast }\left( f_{n}\right) \right\Vert \rightarrow 0$, and hence $%
T\circ S$ is almost limited.

$\left( 2\right) \Rightarrow \left( 3\right) $ Obvious.

$\left( 3\right) \Rightarrow \left( 4\right) $ Let $\left( x_{n}\right)
\subset X$ be a weakly null sequence and let $\left( f_{n}\right) \subset
E^{\ast }$ be a disjoint weak* null sequence. By Theorem 5.26 \cite{AB3},
the operator $S:\ell ^{1}\rightarrow X$ defined by $S\left( \left( \lambda
_{i}\right) \right) =\sum_{i=1}^{\infty }\lambda _{i}x_{i}$, is weakly
compact. Thus, by our hypothesis $T\circ S$ is almost limited and hence $%
\left\Vert \left( T\circ S\right) ^{\ast }\left( f_{n}\right) \right\Vert
\rightarrow 0$. But 
\begin{eqnarray*}
\left\Vert \left( T\circ S\right) ^{\ast }\left( f_{n}\right) \right\Vert
&=&\sup \left\{ \left\vert f_{n}\left( T\left( S\left( \left( \lambda
_{i}\right) \right) \right) \right) \right\vert :\left( \lambda _{i}\right)
\in B_{\ell ^{1}}\right\} \\
&\geq &\left\vert f_{n}\left( T\left( S\left( e_{n}\right) \right) \right)
\right\vert =\left\vert f_{n}\left( T\left( x_{n}\right) \right) \right\vert
\end{eqnarray*}%
for every $n$, where $\left( e_{n}\right) $ is the canonical basis of $\ell
^{1}$. Then $f_{n}\left( T\left( x_{n}\right) \right) \rightarrow 0$, as
desired.

$\left( 4\right) \Rightarrow \left( 1\right) $ Let $W$ be a relatively
weakly compact subset of $X$. If $T\left( W\right) $ is not almost limited
set in $E$ then there exists a disjoint weak* null sequence $\left(
f_{n}\right) \subset E^{\ast }$ such that $\sup \left\{ \left\vert
f_{n}\left( T\left( x\right) \right) \right\vert :x\in W\right\}
\nrightarrow 0$. By choosing a subsequence we may suppose that there is $%
\varepsilon $ with $\sup \left\{ \left\vert f_{n}\left( T\left( x\right)
\right) \right\vert :x\in W\right\} >\varepsilon >0$ for all $n\in \mathbb{N}
$. So for every $n$ there exists some $x_{n}\in W$ with $\left\vert
f_{n}\left( T\left( x_{n}\right) \right) \right\vert \geq \varepsilon $.

On the other hand, since $W$ is a relatively weakly compact subset of $X$,
there exists a subsequence $\left( x_{n_{k}}\right) $ of $\left(
x_{n}\right) $ such that $x_{n_{k}}\overset{w}{\rightarrow }x$ holds in $X$.
By hypothesis, $f_{n_{k}}\left( T\left( x_{n_{k}}-x\right) \right)
\rightarrow 0$ and clearly $f_{n_{k}}\left( T\left( x\right) \right)
\rightarrow 0$. Now from $f_{n_{k}}\left( T\left( x_{n_{k}}\right) \right)
=f_{n_{k}}\left( T\left( x_{n_{k}}-x\right) \right) +f_{n_{k}}\left( T\left(
x\right) \right) $ we see that $f_{n_{k}}\left( T\left( x_{n_{k}}\right)
\right) \rightarrow 0$, which is impossible. Thus, $T\left( W\right) $ is an
almost limited set in $E$, and so $T$ is weak almost limited.
\end{proof}

\begin{remark}
\label{Rem1}Every operator $T:X\rightarrow E$ that admits a factorization
through the Banach lattice $\ell ^{\infty }$, is weak almost limited.

In fact, let $R:X\rightarrow \ell ^{\infty }$ and $S:\ell ^{\infty
}\rightarrow E$ be two operators such that $T=S\circ R$. Let $\left(
x_{n}\right) \subset X$ be a weakly null sequence and let $\left(
f_{n}\right) \subset E^{\ast }$ be a disjoint weak* null sequence. Clearly $%
R\left( x_{n}\right) \overset{w}{\rightarrow }0$ holds in $\ell ^{\infty }$\
and $S^{\ast }f_{n}\overset{w^{\ast }}{\rightarrow }0$ holds in $\left( \ell
^{\infty }\right) ^{\ast }$. Since $\ell ^{\infty }$ has the Dunford-Pettis*
property then $f_{n}\left( Tx_{n}\right) =\left( S^{\ast }f_{n}\right)
\left( R\left( x_{n}\right) \right) \rightarrow 0$. Thus $T$ is weak almost
limited.
\end{remark}

The next result characterizes, under some conditions, the order bounded weak
almost limited operators between two Banach lattices.

\begin{theorem}
\label{regulieraw*DP} Let $E$ and $F$ be two Banach lattices such that the
lattice operations of $E^{\ast }$ are sequentially weak* continuous or $F$
satisfy the property $\mathrm{(d)}$. Then for an order bounded operator $%
T:E\rightarrow F$, the following assertions are equivalents:

\begin{enumerate}
\item $T$ is weak almost limited.

\item For every weakly null sequence $\left( x_{n}\right) \subset E^{+}$ and
every disjoint weak* null sequence $\left( f_{n}\right) \subset F^{\ast }$
we have $f_{n}\left( Tx_{n}\right) \rightarrow 0$.

\item For every disjoint weakly null sequence $\left( x_{n}\right) \subset E$
and every disjoint weak* null sequence $\left( f_{n}\right) \subset F^{\ast
} $ we have $f_{n}\left( Tx_{n}\right) \rightarrow 0$.

\item For every disjoint weakly null sequence $\left( x_{n}\right) \subset
E^{+}$ and every disjoint weak* null sequence $\left( f_{n}\right) \subset
F^{\ast }$ we have $f_{n}\left( Tx_{n}\right) \rightarrow 0$.

\item $T$ carries the solid hull of each relatively weakly compact subset of 
$E$ to an almost limited subset of $F$.
\end{enumerate}

If $F$ has the property $(\text{d})$, we may add:

\begin{enumerate}
\item[$\left( 6\right) $] $f_{n}\left( T\left( x_{n}\right) \right)
\rightarrow 0$ for every weakly null sequence $\left( x_{n}\right) \subset
E^{+}$ and every disjoint weak$^{\ast }$ null sequence $\left( f_{n}\right)
\subset \left( F^{\ast }\right) ^{+}$.

\item[$\left( 7\right) $] $f_{n}\left( T\left( x_{n}\right) \right)
\rightarrow 0$ for every disjoint weakly null sequence $\left( x_{n}\right)
\subset E^{+}$ and every disjoint weak$^{\ast }$ null sequence $\left(
f_{n}\right) \subset \left( F^{\ast }\right) ^{+}$.
\end{enumerate}
\end{theorem}

\begin{proof}
$\left( 1\right) \Rightarrow \left( 2\right) $ and $\left( 1\right)
\Rightarrow \left( 3\right) $ Follows from Theorem \ref{aw*DP}.

$\left( 2\right) \Rightarrow \left( 4\right) $ and $\left( 3\right)
\Rightarrow \left( 4\right) $ are obvious.

$\left( 4\right) \Rightarrow \left( 5\right) $ Let $W$ be a relatively
weakly compact subset of $E$ and let $\left( f_{n}\right) \subset F^{\ast }$
be a disjoint weak* null sequence. Put $A=\mathrm{sol}\left( W\right) $ and
note that if $\left( z_{n}\right) \subset A^{+}:=A\cap E^{+}$ is a disjoint
sequence then by Theorem 4.34 of \cite{AB3} $z_{n}\overset{w}{\rightarrow }0$%
. Thus, by our hypothesis $f_{n}\left( Tz_{n}\right) \rightarrow 0$ for
every disjoint sequence $\left( z_{n}\right) \subset A^{+}$ and every
disjoint weak* null sequence $\left( f_{n}\right) \subset F^{\ast }$. Now,
by Theorem 2.7 of \cite{Mach} $T\left( A\right) $ is almost limited.

$\left( 5\right) \Rightarrow \left( 1\right) $ Obvious.

$\left( 2\right) \Rightarrow \left( 6\right) $ and $\left( 4\right)
\Rightarrow \left( 7\right) $ are obvious.

$\left( 6\right) \Rightarrow \left( 2\right) $ and $\left( 7\right)
\Rightarrow \left( 4\right) $ Let $\left( x_{n}\right) \subset E^{+}$ be a
weakly null (resp. disjoint weakly null) sequence and let $\left(
f_{n}\right) \subset F^{\ast }$ be a disjoint weak* null sequence.

If $F$ has the property $(\text{d})$ then $\left\vert f_{n}\right\vert 
\overset{w^{\ast }}{\rightarrow }0$. So from the inequalities $%
f_{n}^{\,+}\leq \left\vert f_{n}\right\vert $ and $f_{n}^{\,-}\leq
\left\vert f_{n}\right\vert $, the sequences $(f_{n}^{\,+})$, $(f_{n}^{\,-})$
are weak* null. Finally, by $\left( 6\right) $ (resp. $\left( 7\right) $), $%
\lim f_{n}\left( T\left( x_{n}\right) \right) =\lim \left[ f_{n}^{\,+}\left(
T\left( x_{n}\right) \right) -f_{n}^{\,-}\left( T\left( x_{n}\right) \right) %
\right] =0$.
\end{proof}

As consequence of Theorem \ref{regulieraw*DP} we obtain the following
characterization of the wDP* property which is a generalization of Theorem
3.2 of \cite{chen}.

\begin{corollary}
Let $E$ be a Banach lattice with the property $\mathrm{(d)}$. Then the
following assertions are equivalents:

\begin{enumerate}
\item $E$ has the wDP* property.

\item For every weakly null sequence $\left( x_{n}\right) \subset E^{+}$ and
every disjoint weak* null sequence $\left( f_{n}\right) \subset E^{\ast }$
we have $f_{n}\left( x_{n}\right) \rightarrow 0$.

\item For every disjoint weakly null sequence $\left( x_{n}\right) \subset E$
and every disjoint weak* null sequence $\left( f_{n}\right) \subset E^{\ast
} $ we have $f_{n}\left( x_{n}\right) \rightarrow 0$.

\item For every disjoint weakly null sequence $\left( x_{n}\right) \subset
E^{+}$ and every disjoint weak* null sequence $\left( f_{n}\right) \subset
E^{\ast }$ we have $f_{n}\left( x_{n}\right) \rightarrow 0$.

\item The solid hull of every relatively weakly compact set in $E$ is almost
limited.

\item $f_{n}\left( x_{n}\right) \rightarrow 0$ for every weakly null
sequence $\left( x_{n}\right) \subset E^{+}$ and every disjoint weak* null
sequence $\left( f_{n}\right) \subset \left( E^{\ast }\right) ^{+}$.

\item $f_{n}\left( x_{n}\right) \rightarrow 0$ for every disjoint weakly
null sequence $\left( x_{n}\right) \subset E^{+}$ and every disjoint weak*
null sequence $\left( f_{n}\right) \subset \left( E^{\ast }\right) ^{+}$.
\end{enumerate}
\end{corollary}

Recently, the authors in \cite{chen2} demonstrated that if a positive weak*
Dunford-Pettis operator $T:E\rightarrow F$ has its range in $\sigma $%
-Dedekind complete Banach lattice, then every positive operator $%
S:E\rightarrow F$ that it dominates (i.e., $0\leq S\leq T$) is also weak*
Dunford-Pettis \cite[Theorem 3.1]{chen2}. For the positive weak almost
limited operators, the situation still hold when $F$ satisfy the property
(d).

\begin{corollary}
\label{domination}Let $E$ and $F$ be two Banach lattices such that $F$
satisfy the property $\mathrm{(d)}$. Let $S,\ T:E\rightarrow F$ be two
positive operators such that $0\leq S\leq T$. Then $S$ is a weak almost
limited operator whenever $T$ is one.
\end{corollary}

\begin{proof}
Follows immediately from Theorem \ref{regulieraw*DP} by noting that $0\leq
f_{n}\left( Sx_{n}\right) \leq f_{n}\left( Tx_{n}\right) \rightarrow 0$ for
every disjoint weakly null sequence $\left( x_{n}\right) \subset E^{+}$ and
every disjoint weak* null sequence $\left( f_{n}\right) \subset \left(
E^{\ast }\right) ^{+}$.
\end{proof}

Note that, clearly, every almost limited operator $T:X\rightarrow E$, from a
Banach space into a Banach lattice, is weak almost limited. But the converse
is not true in general. Indeed, the identity operator $I:\ell
^{1}\rightarrow \ell ^{1}$ is Dunford-Pettis (and hence weak almost limited)
but it is not almost limited.

The next result characterizes pairs of Banach lattices $E$, $F$ for which
every positive weak almost limited operator $T:E\rightarrow F$ is almost
limited.

\begin{theorem}
\label{awstDPAlimited}Let $E$ and $F$ be two Banach lattices such that $F$
has the property $\mathrm{(d)}$. Then, the following statements are
equivalents:

\begin{enumerate}
\item Every order bounded weak almost limited operator $T:E\rightarrow F$ is
almost limited.

\item Every positive weak almost limited operator $T:E\rightarrow F$ is
almost limited.

\item One of the following assertions is valid:

\begin{enumerate}
\item[\textrm{(i)}] $F$ has the positive dual Schur property.

\item[\textrm{(ii)}] The norm of $E^{\ast }$ is order continuous.
\end{enumerate}
\end{enumerate}
\end{theorem}

\begin{proof}
$\left( 1\right) \Rightarrow \left( 2\right) $ Obvious.

$\left( 2\right) \Rightarrow \left( 3\right) $ Assume by way of
contradiction that $F$ does not have the positive dual Schur property and
the norm of $E^{\ast }$ is not order continuous. We have to construct a
positive weak almost limited operator $T:E\rightarrow F$ which is not almost
limited. To this end, since the norm of $E^{\ast }$ is not order continuous,
there exists a disjoint sequence $(f_{n})\subset \left( E^{\ast }\right)
^{+} $ satisfying $\left\Vert f_{n}\right\Vert =1$ and $0\leq f_{n}\leq f$
for all $n$ and for some $f\in \left( E^{\ast }\right) ^{+}$(see Theorem
4.14 of \cite{AB3}).

On the other hand, since $F$ does not have the positive dual Schur property,
then there is a disjoint weak* null sequence $(g_{n})\subset (F^{\ast })^{+}$
such that $(g_{n})$ is not norm null. By choosing a subsequence we may
suppose that there is $\varepsilon $ with $\left\Vert g_{n}\right\Vert
>\varepsilon >0$ for all $n$. From the equality $\left\Vert g_{n}\right\Vert
=\sup \{g_{n}(y):y\in B_{F}^{+}\}$, there exists a sequence $\left(
y_{n}\right) \subset B_{F}^{+}$ such that $g_{n}(y_{n})\geq \varepsilon $
holds for all $n$.

Now, consider the operators $P:E\rightarrow \ell ^{1}$ and $S:\ell
^{1}\rightarrow F$ defined by 
\begin{equation*}
P(x)=\left( f_{n}(x)\right) _{n}
\end{equation*}%
and 
\begin{equation*}
S\left( (\lambda _{n})_{n}\right) =\sum\limits_{n=1}^{\infty }\lambda
_{n}y_{n}
\end{equation*}%
for each $x\in E$ and each $(\lambda _{n})_{n}\in \ell ^{1}$. Since 
\begin{equation*}
\sum\limits_{n=1}^{N}\left\vert f_{n}\left( x\right) \right\vert \leq
\sum\limits_{n=1}^{N}f_{n}\left( \left\vert x\right\vert \right) =\left(
\vee _{n=1}^{N}f_{n}\right) \left( \left\vert x\right\vert \right) \leq
f\left( \left\vert x\right\vert \right)
\end{equation*}%
for each $x\in E$ and each $N\in \mathbb{N}$, the operator $P$ is well
defined, and clearly $P$ and $S$ are positives. Now, consider the positive
operator $T=S\circ P:E\rightarrow \ell ^{1}\rightarrow F$, and note that $%
T(x)=\sum\limits_{n=1}^{\infty }f_{n}(x)y_{n}$ for each $x\in E$. Clealry,
as $\ell ^{1}$ has the Schur property, then $T$ is Dunford-Pettis and hence $%
T$ is weak almost limited. However, for the disjoint weak* null sequence $%
(g_{n})\subset (F^{\ast })^{+}$, we have for every $n$,%
\begin{equation*}
T^{\ast }\left( g_{n}\right) =\sum\limits_{k=1}^{\infty }g_{n}\left(
y_{k}\right) f_{k}\geq g_{n}\left( y_{n}\right) f_{n}\geq 0\text{.}
\end{equation*}

Thus%
\begin{equation*}
\left\Vert T^{\ast }\left( g_{n}\right) \right\Vert \geq \left\Vert
g_{n}\left( y_{n}\right) f_{n}\right\Vert =g_{n}\left( y_{n}\right) \geq
\varepsilon
\end{equation*}%
for every $n$. This show that $T$ is not almost limited, and we are done.

$(i)\Rightarrow \left( 1\right) $ In this case, every operator $%
T:E\rightarrow F$ is almost limited. In fact,\ let $\left( f_{n}\right)
\subset F^{\ast }$ be a disjoint weak* null sequence. Since $F$ has the
property $(\text{d})$ then the positive disjoint sequence $\left( \left\vert
f_{n}\right\vert \right) \subset F^{\ast }$ is weak* null. So by the
positive dual Schur property of $F$, $\left\Vert f_{n}\right\Vert
\rightarrow 0$, and hence $\left\Vert T^{\ast }\left( f_{n}\right)
\right\Vert \rightarrow 0$, as desired.

$(ii)\Rightarrow \left( 1\right) $ According to Proposition 4.4 of \cite%
{Mach} it is sufficient to show $f_{n}\left( T\left( x_{n}\right) \right)
\rightarrow 0$ for every norm bounded disjoint sequence $\left( x_{n}\right)
\subset E^{+}$ and every disjoint weak* null sequence $\left( f_{n}\right)
\subset \left( F^{\ast }\right) ^{+}$. As the norm of $E^{\ast }$ is order
continuous then every norm bounded disjoint sequence $\left( x_{n}\right)
\subset E^{+}$ is weakly null \cite[Theorem 2.4.14]{MN}. Now, since $T$ is
an order bounded weak almost limited operator then by Theorem \ref%
{regulieraw*DP} $f_{n}\left( T\left( x_{n}\right) \right) \rightarrow 0$ for
every norm bounded disjoint sequence $\left( x_{n}\right) \subset E^{+}$ and
every disjoint weak* null sequence $\left( f_{n}\right) \subset \left(
F^{\ast }\right) ^{+}$. This complete the proof.
\end{proof}

As consequence of Theorem \ref{awstDPAlimited}, we obtain the following
corollary.

\begin{corollary}
For a Banach lattices $E$ with the property $\mathrm{(d)}$, $E^{\ast }$ has
an order continuous norm if and only if every positive weak almost limited
operator $T:E\rightarrow E$ is almost limited.
\end{corollary}

\begin{proof}
Follows from Theorem \ref{awstDPAlimited} by noting that if $E$ has the
positive dual Schur property then the norm of $E^{\ast }$ is order
continuous.
\end{proof}

As the Banach space $\ell ^{1}$ has the Schur property then every operator $%
T:\ell ^{1}\rightarrow E$ is weak almost limited. Another consequence of
Theorem \ref{awstDPAlimited} is the following characterization of the
positive dual Schur property.

\begin{corollary}
A Banach lattices $E$ with the property $\mathrm{(d)}$, has the positive
dual Schur property if and only if every positive operator $T:\ell
^{1}\rightarrow E$ is almost limited.
\end{corollary}

Recall that a reflexive Banach space with the Dunford-Pettis property is
finite dimensional \cite[Theorem 5.83]{AB3}. We can prove a similar result
for Banach lattices as follows.

\begin{proposition}
Let $E$ be a Banach lattice with the wDP* property. If the norms of $E$ and $%
E^{\ast }$ are order continuous then $E$ is finite dimensional.

In particular, a reflexive Banach lattice with the wDP* property is finite
dimensional.
\end{proposition}

\begin{proof}
Assume that the norms of $E$ and $E^{\ast }$ are order continuous. As $E$
has the wDP* property the identity operator $I:E\rightarrow E$ is weak
almost limited. Since the norm of $E^{\ast }$ is order continuous then, by
Theorem \ref{awstDPAlimited}, $I$ is almost limited. So $E$ has the positive
dual Schur property. Now, as the norm of $E$ is order continuous then $E$ is
finite dimensional \cite[Proposition 2.1 (b)]{WN2012}. For the second part,
it is enough to note that if $E$ is a reflexive Banach lattice then the
norms of $E$ and $E^{\ast }$ are order continuous \cite[Theorem 4.70]{AB3}.
\end{proof}

Note that from Theorem \ref{regulieraw*DP}, it is easy to see that if $F$ is
a Banach lattice with property $\mathrm{(d)}$ then every order bounded
almost Dunford-Pettis operator $T:E\rightarrow F$ is weak almost limited.
But the convers is false in general. In fact, the identity operator $I:\ell
^{\infty }\rightarrow \ell ^{\infty }$ is weak almost limited operator but
it fail to be almost Dunford-Pettis.

The following result characterizes pairs of Banach lattices $E$, $F$ for
which every positive weak almost limited operator operator $T:E\rightarrow F$
is almost Dunford-Pettis.

\begin{theorem}
\label{aw*DP-aDP}Let $E$ and $F$ be two Banach lattices such that $F$ is $%
\sigma $-Dedekind complete. Then, the following statements are equivalents:

\begin{enumerate}
\item Every order bounded weak almost limited operator $T:E\rightarrow F$ is
almost Dunford-Pettis.

\item Every positive weak almost limited operator $T:E\rightarrow F$ is
almost Dunford-Pettis.

\item One of the following assertions is valid:

\begin{enumerate}
\item[\textrm{(i)}] $E$ has the positive Schur property.

\item[\textrm{(ii)}] The norm of $F$ is order continuous.
\end{enumerate}
\end{enumerate}
\end{theorem}

\begin{proof}
$\left( 1\right) \Rightarrow \left( 2\right) $ Obvious.

$\left( 2\right) \Rightarrow \left( 3\right) $ Assume by way of
contradiction that $E$ does not have the positive Schur property and the
norm of $F$ is not order continuous. We have to construct a positive weak
almost limited operator $T:E\rightarrow F$ which is not almost
Dunford-Pettis. As $E$ does not have the positive Schur property, then there
exists a disjoint weakly null sequence $(x_{n})$ in $E^{+}$ which is not
norm null. By choosing a subsequence we may suppose that there is $%
\varepsilon $ with $\left\Vert x_{n}\right\Vert >\varepsilon >0$ for all $n$%
. From the equality $\left\Vert x_{n}\right\Vert =\sup \{f(x_{n}):f\in
(E^{\ast })^{+},\quad \left\Vert f\right\Vert =1\}$, there exists a sequence 
$\left( f_{n}\right) \subset (E^{\ast })^{+}$ such that $\left\Vert
f_{n}\right\Vert =1$ and $f_{n}(x_{n})\geq \varepsilon $ holds for all $n$.

Now, consider the operator $R:E\rightarrow \ell ^{\infty }$ defined by 
\begin{equation*}
R(x)=(f_{n}(x))_{n}
\end{equation*}%
On the other hand, since the norm of $F$ is not order continuous, it follows
from Theorem 4.51 of \cite{AB3} that $\ell ^{\infty }$ is lattice embeddable
in $F$, i.e., there exists a lattice homomorphism $S:\ell ^{\infty
}\rightarrow F$ and there exist tow positive constants $M$ and $m$ satisfying%
\begin{equation*}
m\left\Vert \left( \lambda _{k}\right) _{k}\right\Vert _{\infty }\leq
S\left( \left( \lambda _{k}\right) _{k}\right) \leq M\left\Vert \left(
\lambda _{k}\right) _{k}\right\Vert _{\infty }
\end{equation*}%
for all $\left( \lambda _{k}\right) _{k}\in \ell ^{\infty }$. Put $T=S\circ
R $, and note that $T$ is a positive weak almost limited operator (see
Remark \ref{Rem1}). However, for the disjoint weakly null sequence $%
(x_{n})\subset E^{+}$, we have%
\begin{equation*}
\left\Vert T\left( x_{n}\right) \right\Vert =\left\Vert S\left( \left(
f_{k}\left( x_{n}\right) \right) _{k}\right) \right\Vert \geq m\left\Vert
\left( f_{k}\left( x_{n}\right) \right) _{k}\right\Vert _{\infty }\geq
mf_{n}\left( x_{n}\right) \geq m\varepsilon
\end{equation*}%
for every $n$. This show that $T$ is not almost Dunford-Pettis, and we are
done.

$\left( i\right) \Rightarrow \left( 1\right) $ In this case, every operator $%
T:E\rightarrow F$ is almost Dunford-Pettis.

$\left( ii\right) \Rightarrow \left( 1\right) $ Let $T:E\rightarrow F$ be an
order bounded weak almost limited operator and let $\left( x_{n}\right)
\subset E$ be a positive disjoint weakly null sequence. We shall show that $%
\left\Vert Tx_{n}\right\Vert \rightarrow 0$. By corollary 2.6 of \cite{DF},
it suffices to proof that $\left\vert Tx_{n}\right\vert \overset{w}{%
\rightarrow }0$ and $f_{n}\left( Tx_{n}\right) \rightarrow 0$ for every
disjoint and norm bounded sequence $\left( f_{n}\right) \subset \left(
F^{\ast }\right) ^{+}$.

Indeed

- Let $f\in \left( F^{\ast }\right) ^{+}$. By Theorem 1.23 of \cite{AB3},
for each $n$ there exists some $g_{n}\in \left[ -f,f\right] $ with $%
f\left\vert Tx_{n}\right\vert =g_{n}\left( Tx_{n}\right) $. Note that the
adjoint operator $T^{\ast }:F^{\ast }\rightarrow E^{\ast }$ is order bounded 
\cite[Theorem 1.73]{AB3}, and pick some $h\in \left( E^{\ast }\right) ^{+}$
with $T^{\ast }\left[ -f,f\right] \subseteq \left[ -h,h\right] $. So $0\leq f\left\vert Tx_{n}\right\vert =\left( T^{\ast }g_{n}\right)
\left( x_{n}\right) \leq h\left( x_{n}\right) $ for all $n$. Since $x_{n}%
\overset{w}{\rightarrow }0$, then $h\left( x_{n}\right) \rightarrow 0$ and
hence $f\left\vert Tx_{n}\right\vert \rightarrow 0$. Thus $\left\vert
Tx_{n}\right\vert \overset{w}{\rightarrow }0$.

- Let $\left( f_{n}\right) \subset \left( F^{\ast }\right) ^{+}$ be a
disjoint and norm bounded sequence. As the norm of $F$ is order continuous,
then by corollary 2.4.3 of \cite{MN} $f_{n}\overset{w^{\ast }}{\rightarrow }%
0 $. Now, since $T$ is an order bounded weak almost limited then $%
f_{n}\left( Tx_{n}\right) \rightarrow 0$ (Theorem \ref{regulieraw*DP}). This
complete the proof.
\end{proof}

As\ consequence of Theorem \ref{aw*DP-aDP}, we obtain the following
corollary.

\begin{corollary}
A $\sigma $-Dedekind complete Banach lattices $E$ has an order continuous
norm if and only if every order bounded weak almost limited operator $%
T:E\rightarrow E$ is almost Dunford-Pettis.
\end{corollary}

By Remark \ref{Rem1} every operator $T:E\rightarrow \ell ^{\infty }$ is weak
almost limited. As another consequence of Theorem \ref{aw*DP-aDP}, we obtain
the following characterization of the positive Schur property.

\begin{corollary}
A Banach lattices $E$ has the positive Schur property if and only if every
positive operator $T:E\rightarrow \ell ^{\infty }$ is almost Dunford-Pettis.
\end{corollary}


\begin{thebibliography}{99}
\bibitem{AB} Aliprantis, C.D. and Burkinshaw, O.: \emph{Dunford-Pettis
operators on Banach lattices}, Trans. Amer. Math. Soc. \textbf{274} (1982),
227-238.

\bibitem{AB3} Aliprantis, C.D. and Burkinshaw, O.: \emph{Positive operators}%
. Reprint of the 1985 original. Springer, Dordrecht, 2006.

\bibitem{Bouras} Bouras, K.: \emph{Almost Dunford--Pettis sets in Banach
lattices}. Rend. Circ. Mat. Palermo (2013) \textbf{62}:227--236.

\bibitem{Bouras2} Bouras, K. and Moussa, M.: Banach lattices with weak
Dunford-Pettis property, World Academy of Science, Engineering and
Technology. International Journal of Engineering and Mathematical Science,
50 (2011), 773--777.

\bibitem{chen} Chen J.X., Chen Z.L., Ji G.X.: \emph{Almost limited sets in
Banach lattices,} J. Math. Anal. Appl. \textbf{412} (2014) 547--553.

\bibitem{chen2} Chen J.X., Chen Z.L., Ji G.X.: \emph{Domination by positive
weak* Dunford-Pettis operators on Banach lattices}, %
\url{http://arxiv.org/abs/1311.2808}

\bibitem{DF} Dodds P.G. and Fremlin D.H.: \emph{Compact operators on Banach
lattices}, Israel J. Math. \textbf{34} (1979) 287-320.

\bibitem{Hmi} H'michane J., El Kaddouri A., Bouras K., Moussa M.: \emph{On
the class of limited operators}, Acta Math. Sci., to appear.

\bibitem{Kad} El Kaddouri A., H'michane J., Bouras K., Moussa M.: \emph{On
the class of weak* Dunford-Pettis operators}, Rend. Circ. Mat. Palermo (2) 
\textbf{62} (2013), 261-265.

\bibitem{Mach} Machrafi N., Elbour A., Moussa M.: \emph{Some
characterizations of almost limited sets and applications}. %
\url{http://arxiv.org/abs/1312.2770}

\bibitem{MN} Meyer-Nieberg P.: \emph{Banach lattices}. Universitext.
Springer-Verlag, Berlin, 1991.

\bibitem{WN3} Wnuk W.: \emph{Banach lattices with the weak Dunford-Pettis
property}. Atti Sem. Mat. Fis. Univ. Modena 42 (1994), no. \textbf{1},
227--236.

\bibitem{WN2012} Wnuk W.: \emph{On the dual positive Schur property in
Banach lattices. }Positivity (2013) 17:759--773.
\end{thebibliography}
\end{document}